\documentclass[11pt,twoside]{amsart}
\usepackage[all]{xy}   
\usepackage {amssymb,latexsym,amsthm,amsmath,mathtools}
\usepackage{enumitem,color,hyperref}

\long\def\symbolfootnote[#1]#2{\begingroup
\def\thefootnote{\fnsymbol{footnote}}\footnote[#1]{#2}\endgroup}

\newcommand{\C}{\mathfrak C}

\makeatletter
\def\imod#1{\allowbreak\mkern10mu({\operator@font mod}\,\,#1)}
\makeatother

\newtheorem{theorem}{Theorem}[section]
\newtheorem{lemma}[theorem]{Lemma}

\newtheorem{proposition}[theorem]{Proposition}
\newtheorem*{theorem*}{Theorem}
\theoremstyle{definition}

\newtheorem{remark}[theorem]{Remark}

\numberwithin{equation}{section}

\newcommand{\ignore}[1]{}

\newcommand{\mynote}[1]{}

\begin{document}

\setcounter{section}{0}

\title{A Frobenius-type formula for compact Lie groups}
\author{Shripad M. Garge}
\address{Department of Mathematics, Indian Institute of Technology, Powai, Mumbai, 400 076 India.}
\email{shripad@math.iitb.ac.in}
\email{smgarge@gmail.com}
\author{Uday Bhaskar Sharma}
\address{School of Liberal Studies and Media, UPES, Kandoli, Dehradun - 248 007, India.}
\email{udaybsharmaster@gmail.com}

\subjclass[2010]{20G15, 20D06}
\keywords{compact Lie groups, Frobenius formula, character theory, commutator probability}

\begin{abstract}
Let $G$ be a group and $\alpha: G \times G \to G$ denote the commutator map. 
In the case of finite groups, Frobenius gave the formula to compute the cardinalities of the fibres $\alpha^{-1}(g)$ in terms of the character values $\chi(g)$ for irreducible characters $\chi$ of $G$.
We generalise this formula to compact Lie groups. 
Further, we connect this generalised formula to the commutator probability of the concerned groups. 
\end{abstract}

\maketitle

\section{Introduction}

Let $G$ be a group and let $\alpha: G \times G \to G$ be the commutator map, $\alpha(g, h) = ghg^{-1}h^{-1}$.
We are interested in studying the fibres $\alpha^{-1}(g)$ as $g$ varies in $G$.

In the case of finite groups, let $f: G \to \mathbb{C}$ denote the function that assigns to each $g \in G$ the number of elements in the fibre $\alpha^{-1}(g)$. 
Then $f$ is a class function on $G$, that is, it is constant on conjugacy classes of $G$, and hence it is a complex linear combination of irreducible characters of $G$. 
Frobenius discovered the following formula for $f$:
$$f(g) = |G| \sum_{\chi} \frac{\chi(g)}{\chi(1)} $$
where $\chi$ varies over all irreducible characters of $G$. 

In this paper, we generalise this formula to compact Lie groups. 
A compact group comes equipped with a unique normalised left-translation invariant measure, called the Haar measure. 
If $\mu$ denotes the Haar measure on a compact $G$ then the function defined by $g \mapsto \mu\left(\alpha^{-1}(g)\right)$ is a class function on $G$. 
This is our analogue of the function $f$ defined above and we would like to get a formula, {\it \`{a} la Frobenius}, for this function in terms of irreducible characters of $G$. 

There are some natural restrictions that we need to consider. 
First, among the compact groups, the representation theory of compact {\it Lie} groups is reasonably well-understood. 
Hence we derive the formula for compact Lie groups. 
There are, in general, infinitely many irreducible characters of such a group $G$. 
Hence we will need to make sense of the sum in the generalised formula. 
We consider it as a limit over certain finite subsets as the size of the finite subsets goes to infinity. 
We prove the result for the connected compact Lie group and the FC compact Lie groups. 
Each finite group is an FC compact Lie group hence our result is a generalisation of the beautiful formula of Frobenius. 

Let $G$ be a connected compact Lie group and $T$ be a maximal torus in $G$. 
Since $G$ is a union of conjugates of $T$, each irreducible character of $G$ is determined by its restriction on $T$. 
As characters of $T$, these characters are associated with finite sequences of non-negative integers, called weights. 
We will take the limit in the generalised formula as the sum of the integers in these weights goes to $\infty$. 
We refer the reader to section 3 for a more detailed discussion on this.
Let $Irr(G)_n$ denote the set of irreducible characters of $G$ whose weight-sum is less than or equal to $n$. 
Then we have

\begin{theorem}
Let $G$ be a connected, compact, Lie group with normalised Haar measure $\mu$.
Let $d$ denote the dimension of the group $G$ and let $r$ denote the rank of $G$. 
Let $Irr(G)_n$ denote the set of irreducible characters of $G$, whose weight-sum is less than or equal to $n$. 
Then for each $g \in G$, 
$$\mu\left(\alpha^{-1}(g)\right) = \lim_{n \to \infty} \frac{1}{|Irr(G)_n|^{d-r+1}} \sum_{\chi \in Irr(G)_n} \frac{\chi(g)}{\chi(1)} .$$
\end{theorem}

In the next section, we recall the proof in the case of finite groups and prove the above theorem in the third section. 

A group is called an FC group if each of its conjugacy classes is finite. 
The finite groups and compact abelian groups are examples of FC groups. 
We prove in the fourth section that a compact Lie FC group is essentially built up from finite groups and compact abelian groups. 
We prove the following general version of the above theorem for compact FC Lie groups. 
\begin{theorem}
Let ${\displaystyle G \cong \frac{G^0\times \Delta}{N}}$ be a compact Lie FC group. 
Then for $g_0\in G$, we have 
$$\mu^2 \big(\alpha^{-1}(g_0)\big) = \frac{|N|}{|\Delta|}\lim_{t \rightarrow \infty}\left(\frac{1}{|Irr(G^0)_t|}\sum_{\chi \in Irr(G)_t}\frac{\chi(g_0)}{\chi(1)}\right).$$
\end{theorem}

In the next section, we deal with the remaining class of compact Lie groups for which the theorem is not yet discussed. 
These are the compact Lie groups which are not FC groups. 
Among these, we look at only those groups $G$ such that the FC centre of $G$ is an open subgroup of $G$. 
This is because these are the only groups for which the trivial fibre of the commutator map has a positive measure.
We then prove that each fibre of the commutator map for $G$ has the same measure as that of the commutator map for the FC centre. 

In the last section, we connect Theorem 1.1 with the commutator probability.

\section*{Acknowledgements}
We thank Dipendra Prasad for many useful discussions. 

\section{Finite groups}
We recall the proof of Frobenius' theorem in this section.
We fix a finite group $G$ for this section. 
Let us begin with a lemma.
\begin{lemma}
For $g, h\in G$, and an irreducible character $\chi$ of $G$ 
$$\chi(g)\chi(h) = \frac{\chi(1)}{|G|}\sum_{z\in G}\chi(gzhz^{-1}).$$
\end{lemma}
\begin{proof}
Let $\rho$ be the irreducible representation of $G$, whose character is $\chi$, and let $V$ be the representation space. 
For a fixed $h \in G$, we define an endomorphism $f_h$ of $V$ as follows:
$$f_h = \frac{\chi(1)}{|G|} \sum_{z\in G} \rho(zhz^{-1}) .$$ 
One sees that $f_h\rho(t) = \rho(t)f_h$ for every $t \in G$. 
By Schur's lemma, $f_h = \lambda_h I_V$, for some constant $\lambda_h \in \mathbb{C}$. 
One computes the constant $\lambda_h$ by taking the trace of $f_h$ which gives:
$$Tr(f_h) = \frac{\chi(1)}{|G|} \sum_{z \in G} \chi(zhz^{-1}) = \chi(1)\chi(h) .$$
On the other hand $Tr(f_h) = \lambda_h \chi(1)$, hence we get $\lambda_h = \chi(h)$ and $f_h = \chi(h) I_V$.

Further, we consider the endomorphism $\rho(g)f_h$ of $V$ whose trace computation gives 
$$\chi(h)\chi(g) = Tr\left(\rho(g)f_h\right) = \frac{\chi(1)}{|G|} \sum_{z \in G} \chi(gzhz^{-1}) .$$
\end{proof}

\begin{theorem}
Let $G$ be a finite group and let $f(g)$ denote the number of elements $(x, y) \in G \times G$ with $[x, y] = g$. 
Then 
$$f(g) = |G| \sum_{\chi} \frac{\chi(g)}{\chi(1)}$$
where $\chi$ varies over the set of irreducible characters of $G$. 
\end{theorem}
\begin{proof}
We have $f(g) =|\{(x,y) \in G^2\mid [x,y] = g\}|$, the cardinality of the fibre $\alpha^{-1}(g)$ where $\alpha: G \times G \to G$ is the commutator map. 
The function $f$ is a class function, hence it is a complex linear combination of irreducible characters of $G$. 
Let $f = \sum_{\chi} s_\chi \chi$ where 
$$s_\chi = \frac{1}{|G|}\sum_{g\in G} f(g)\chi(g) .$$ 
Now, from the lemma above, we have 
$$\chi(x)\chi(x^{-1})= \frac{\chi(1)}{|G|}\sum_{z\in G}\chi(xzx^{-1}z^{-1}) .$$ 
With $\chi(x^{-1}) = \overline{\chi(x)}$ for any $x\in G$, we have:
$$|G| = \sum_{x \in G}\chi(x)\chi(x^{-1}) = \frac{\chi(1)}{|G|}\sum_{x,z\in G}\chi([x,z])$$
$$= \frac{\chi(1)}{|G|}\sum_{g\in G} f(g)\chi(g) 
= \chi(1)s_\chi .$$
This gives $$s_\chi = \frac{|G|}{\chi(1)}$$ 
and hence 
$$|\{(x,y) \in G^2 \mid [x,y] = g\}| = f(g) = |G|\sum_{\chi \in \hat{G}}\frac{\chi(g)}{\chi(1)} .$$
\end{proof}

\section{Connected compact Lie groups}

In this section, we prove the main theorem for connected compact Lie groups. 
It turns out that if a connected compact Lie group $G$ is not abelian then each fibre of the commutator map has measure zero. 
We prove that both the sides of the equality in the main theorem equal zero in this case and that proves that the theorem holds in this case. 

We first prove a basic lemma in the general case of compact groups. 
\begin{lemma}
Let $G$ be a compact group with the normalised Haar measure $\mu$. 
Let $\alpha: G \times G \to G$ denote the commutator map. 
If $g$ and $h$ are conjugate elements of $G$ then $\mu(\alpha^{-1}(g)) = \mu(\alpha^{-1}(h))$. 
\end{lemma}

\begin{proof}
Let $G$ be a compact group and $g, h \in G$ be conjugate elements. 
If $g = a^{-1}ha$ then $(g_1, g_2) \in \alpha^{-1}(g)$ if and only if $(ag_1a^{-1}, ag_2a^{-1}) \in \alpha^{-1}(h)$. 
Since inner automorphisms preserve Haar measure of (measurable) sets, we get $\mu(\alpha^{-1}(g)) = \mu(\alpha^{-1}(h))$.
\end{proof}

Now, we turn to the case of connected compact groups. 
We are still in a reasonably general setup as we are not restricting to Lie groups yet. 

\begin{proposition}\label{PropNonCent}
Let $G$ be a connected, compact group with the normalised Haar measure $\mu$. 
If $g \in G$ is a non-central element then the $\mu(\alpha^{-1}(g)) = 0$.
\end{proposition}

\begin{proof}
As $g$ is a non-central element of $G$ its centraliser, $Z_G(g)$, is a proper subgroup of $G$.
Since $G$ is connected, the dimension of $Z_G(g)$ is strictly less than that of $G$ and hence the dimension of the conjugacy class of $g$ is non-zero, being equal to $\dim(G) - \dim(Z_G(g))$. 
In particular, $g$ has infinitely many conjugates in $G$. 

If $h$ is a conjugate of $g$ then, by the above lemma, $\mu(\alpha^{-1}(g)) = \mu(\alpha^{-1}(h))$. 
The sets $\alpha^{-1}(h)$, as $h$ varies over conjugates of $g$, are all disjoint and their union is a subset of $G$, hence its measure is finite.
If $\mu(\alpha^{-1}(g))$ is non-zero, then we get a contradiction to the finiteness of $\mu \left(\cup_h \alpha^{-1}(h)\right)$. 
Hence $\mu(\alpha^{-1}(g)) = 0$.
\end{proof}

\begin{remark}
We note that the connectedness of $G$ is a necessary condition in the proposition above. 
If $G$ has a finite number of connected components then every element of $G$ may have only finitely many conjugates. 
Take, for instance, the group $G = S^1 \times H$ where $H$ is a finite {\em non-abelian} group. 
The connected component of identity, $S^1$, is also a central subgroup, so $S^1 \subseteq Z_G(g)$ for every $g \in G$. 
Therefore $\dim(Z_G(g)) = \dim(G)$ for any $g \in G$ and each conjugacy class in $G$ has dimension zero. 
If $g = (t, h) \in S^1 \times H$ then the conjugacy class of $g$ in $G$ is determined by the conjugacy class of $h$ in $H$ and is hence finite. 
\end{remark}

We now turn our attention to central elements of compact, connected Lie groups $G$. 

\begin{proposition}
Let $g \in SU(n)$ be a central element. 
Then $\mu(\alpha^{-1}(g)) = 0$. 
\end{proposition}

\begin{proof}
Let $a, b \in SU(n)$ such that $aba^{-1}b^{-1} = g$. 
We note that since $g$ is a central element, $g = \zeta I$ where $\zeta$ is a root of unity. 
Further, again since $g$ is central, we can modify $(a, b)$ by simultaneous conjugation and thus assume that $a$ is diagonal with diagonal entries $a_1, \dots, a_n$. 
If $b = [b_{ij}]$ then the equality $aba^{-1} = \zeta b$ gives us $a_i a_j^{-1} b_{ij} = \zeta b_{ij}$. 
In particular, we get that $a_i= \zeta a_j$ whenever $b_{ij} \ne 0$. 

If $\zeta \ne 1$ then all $b_{ii}$ are zero. 
Further, up to a permutation of $\{1, \dots, n\}$, $a_{i+1} = \zeta a_i$. 
Since the determinant of $a$ is equal to $1$, we get that 
$$a_1 a_2 \cdots a_n = \zeta^{n(n-1)/2}a_1^n = 1 .$$
This implies that the eigenvalues of $a$ have a finite choice and hence the set $\alpha^{-1}(g)$ has measure zero. 

If $\zeta = 1$ then $a$ is a scalar matrix whose eigenvalues are $n$-th roots of unity and hence the set $\alpha^{-1}(g)$ has measure zero. 
\end{proof}

\begin{proposition}
Let $G$ be a non-abelian, connected, compact Lie group and let $g \in G$ be a central element. 
The set $\alpha^{-1}(g)$ has measure zero in $G$. 
\end{proposition}

\begin{proof}
By going to a finite cover if necessary, we can assume that $G$ is a direct product of a torus and simple Lie groups which can all be assumed to be compact, connected and simply connected. 

The restriction of $\alpha$ to each of these direct factors is the commutator map of those direct factors and the measure of the set $\alpha^{-1}(g)$ is the product of measures of inverse images for each direct factor. 
It is, therefore, enough, to show that the inverse image has measure zero under any one direct factor. 
We will show that this holds for a non-abelian direct factor of $G$. 

We now assume that $G$ is one of these non-abelian direct factors. 
Any such group has a faithful irreducible representation $\rho: G \hookrightarrow SU(n)$ for some $n$ and then by Schur's lemma the image of a central element of $g \in G$ under this representation is a central element in $SU(n)$. 

The commutator map $\alpha$ on $G \times G$ is the restriction of the commutator map on $SU(n) \times SU(n)$ and $\alpha^{-1}(g)$ is simply the inverse of $g$ under the commutator map of $SU(n)$ intersected with $G$. 

From the above lemma, and its proof, we see if $a, b \in G$ with $aba^{-1}b^{-1} = g$ then the eigenvalues of $a$ are of a special form and hence the measure of $\alpha^{-1}(g)$ is zero. 
\end{proof}

We start with a connected, compact, Lie group and assume that it is non-abelian. 
By going to a finite cover, if necessary, we may (and will) assume that $G$ is a direct product of a torus and connected, simply connected, simple Lie groups. 
Since $G$ is assumed to be non-abelian, it will have at least one direct factor which is not a torus. 

It is known that irreducible representations of a compact Lie group $G$ are parametrized by dominant integral weights of a (fixed) maximal torus $T$, which are integral linear combinations of the fundamental weights of $T$. 
Thus, irreducible representations of $G$ are parametrized by ordered $r$-tuples of non-negative integers, $(n_1, \dots, n_r)$, where $r$ is the rank of the group $G$ which is equal to the dimension of $T$.
We let $Irr(G)_n$ denote the set of all irreducible representations of $G$ with the property that the sum of the corresponding integers is less than or equal to $n$. 

To give an example, if $G$ is $SU_2$ then the irreducible representations of $G$ are parametrized by weights of $S^1$, non-negative integers, and hence $Irr(SU_2)_n$ is the same as the set $\{0, 1, \dots, n\}$. 

We now state and prove our main theorem in this section. 

\begin{theorem}
Let $G$ be a connected, compact, non-abelian, Lie group.
Let $d$ denote the dimension of the group $G$ and let $r$ denote its rank. 
Let $Irr(G)_n$ denote the set of irreducible characters of $G$ whose sum of the integers in the corresponding weight is less than or equal to $n$, as defined above. 
Then for each $g \in G$, 
$$\lim_{n \to \infty} \frac{1}{|Irr(G)_n|^{d-r+1}} \sum_{\chi \in Irr(G)_n} \frac{\chi(g)}{\chi(1)} = 0.$$
\end{theorem}

\begin{proof}
By going to a finite cover, if necessary, we will assume that $G$ is a direct product of a torus and simple compact, connected Lie groups. 
The character values, $\chi(g)$, are certain finite sums of roots of unity and hence there is a natural upper bound on $|\chi(g)|$, the dimension $\chi(1)$ of the corresponding representation. 
Hence we get
$$\left|\sum_{\chi \in Irr(G)_n} \frac{\chi(g)}{\chi(1)}\right| \leq \sum_{\chi \in Irr(G)_n} \left|\frac{\chi(g)}{\chi(1)}\right| \leq \left|Irr(G)_n\right| .$$
Then 
$$\left|\frac{1}{|Irr(G)_n|^{d-r+1}} \sum_{\chi \in Irr(G)_n} \frac{\chi(g)}{\chi(1)}\right| \leq \frac{1}{|Irr(G)_n|^{d-r}} .$$
Since our group $G$ is non-abelian, it follows that it must have at least one non-abelian simple direct factor. 
In particular, it is not a torus and hence $d-r > 1$. 
Now the proof follows. 
\end{proof}

\begin{proposition}
Let $T$ be a connected, compact, abelian, Lie group.
Let $Irr(T)_n$ denote the set of irreducible characters of $T$ whose sum of the integers in the corresponding weight is less than or equal to $n$, as defined above. 
Then for each non-trivial $g \in T$, 
$$\lim_{n \to \infty} \frac{1}{|Irr(T)_n|} \sum_{\chi \in Irr(T)_n} \frac{\chi(g)}{\chi(1)} = 0.$$
\end{proposition}

\begin{proof}
A compact, connected, abelian Lie group is a compact torus, if $k = \dim T$ then $T$ is isomorphic to $k$ copies of $S^1$. 
Further, every irreducible character $\chi$ of $T$ is one dimensional and 
$$\chi(e^{i\theta_1}, \dots, e^{i\theta_k}) = \prod_j e^{im_j\theta_j} \hskip2mm {\rm with} \hskip2mm m_j \in \mathbb{Z} .$$
The weight of $\chi$ is given by the $k$-tuple $(m_1,\ldots, m_k)$. 
Thus, the set of characters of $T$ with weight $\leq n$ are given by $(m_1,\ldots, m_k)$ with $|m_j| \leq n$ for all $j$. 
The number of characters with this property is $(2n+1)^k$. 

Now, for $g \in T$, where $g = (e^{i\theta_1},\ldots, e^{i\theta_k})$ we want to compute
$$\lim_{n \to \infty} \frac{1}{|Irr(T)_n|} \sum_{\chi \in Irr(T)_n} \frac{\chi(g)}{\chi(1)} .$$

If $\theta_j \in 2\pi\mathbb{Z}$ for all $j$ then $g = 1$ and we get the above limit to be $1$.

On the other hand, if $g$ is non-trivial then some $\theta_j \not \in 2\pi\mathbb{Z}$. 
Hence
$$\lim_{n \to \infty} \frac{1}{|Irr(T)_n|} \sum_{\chi \in Irr(T)_n} \frac{\chi(g)}{\chi(1)} =  \lim_{n \to \infty} \frac{1}{(2n+1)^k}\sum_{|m_j|\leq n} \prod_{j=1}^ke^{im_j\theta_j}$$
$$= \lim_{n\rightarrow \infty} \frac{1}{(2n+1)^k}\prod_{j=1}^k\left(\sum_{m_j =- n}^n e^{im_j\theta_j}\right)$$

Now, since $\theta_j \not\in 2\pi\mathbb{Z}$, ${\displaystyle \sum_{m_j =- n}^n e^{im_j\theta_j} = 1+ \sum_{m_j = 1}^n 2\cos (m_j\theta_j) = \frac{\sin ((n+ 1/2)\theta_j)}{\sin(\theta_j/2)}}$ which is bounded above for the fixed $\theta_j$. 
Hence the above limit is zero. 
\end{proof}

We recall the statement of Theorem 1.1 here and give its proof.
\vskip5mm 

\noindent{\bf Theorem 1.1.}
{\em 
Let $G$ be a connected, compact, Lie group with normalised Haar measure $\mu$.
Let $d$ denote the dimension of the group $G$ and let $r$ denote the rank of $G$. 
Let $Irr(G)_n$ denote the set of irreducible characters of $G$ whose weight-sum is less than or equal to $n$. 
Then for each $g \in G$, 
$$\mu\left(\alpha^{-1}(g)\right) = \lim_{n \to \infty} \frac{1}{|Irr(G)_n|^{d-r+1}} \sum_{\chi \in Irr(G)_n} \frac{\chi(g)}{\chi(1)} .$$}

\begin{proof}
If $G$ is an abelian, compact, connected, Lie group then it is a compact torus and we have $d = r$. 
If $g \in G$ is the identity element then $\alpha^{-1}(g)$ is the whole $G \times G$, hence $\mu\left(\alpha^{-1}(g)\right) = 1$. 
It follows, as mentioned in the proof of the above proposition, that the right-hand side also gives $1$. 

We now prove that, in all other cases, both sides of the required equation are zero. 

If $G$ is abelian and $g \in G$ is non-trivial then $\alpha^{-1}(g)$ is the empty set and its measure is zero.
The limit of the character sum is also zero in this case, by Proposition 3.7. 

If $G$ is non-abelian then the limit of the character sum in the above equation is zero for every $g \in G$ by Theorem 3.6. 
If $G$ is non-abelian and $g \in G$ is a non-central element then $\mu(\alpha^{-1}(g)) = 0$ by Proposition 3.2 and finally if $g$ is a central element of a non-abelian $G$ then $\mu(\alpha^{-1}(g)) = 0$ by Proposition 3.5. 

This completes the proof. 
\end{proof}

\section{FC Compact Groups}

Let $G$ be a group. 
The {\em FC-centre} of $G$, denoted by $FC(G)$,  is defined to be the set of $g \in G$ such that the conjugacy class $\mathcal{C}_G(g)$ of $g$ in $G$ is finite: 
$$FC(G) = \{g \in G: |\mathcal{C}_G(g)|<\infty \} .$$
The FC-centre of a group $G$ is a normal subgroup of $G$. 
A group $G$ is said to be an {\em FC-group} if $FC(G) = G$.
The centre of a group $G$ is a subgroup of $FC(G)$ and every abelian group is an FC-group. 
There are examples of non-abelian FC groups, for instance, take $G$ to be the normaliser of the maximal torus in $SU(n)$ for $n > 1$. 

Let $G$ be a compact Lie FC group of the form $G = A\times T$ where $A$ is a finite group and $T$ is a torus. 
Let $Irr(T)_n$ denote the set of all irreducible characters of $T$ of weight less than or equal to $n$.
Let $\mu$ denote the normalised Haar measure on $G$ and $\mu^2$ denote the corresponding measure on $G \times G$. 
Finally, let $\alpha: G \times G \to G$ denote the commutator map. 
We then have:

\begin{lemma}\label{LemSnFin}
With the above notations, let $X_0\in G$. 
Then 
$$\mu^2 \big(\alpha^{-1}(X_0)\big)  = \lim_{n\rightarrow\infty}\left(\frac{1}{|A||Irr(T)_n|}\sum_{\chi \in Irr(A\times T)_n} \frac{\chi(X_0)}{\chi(1)}\right) .$$
\end{lemma}
\begin{proof}
Let $X_0 = (a_0, g_0)\in G = A \times T$. 
Then the set $G^{(2)}_{X_0} := \{(X,Y)\in G^2 \mid [X,Y] = X_0 \}$ is of the form:
$\big\{((a_1,g_1), (a_2,g_2))\in A \times T \mid [a_1,a_2] = a_0, [g_1, g_2] = g_0 \big\}$.
Thus $G^{(2)}_{X_0} = A^{(2)}_a \times T^{(2)}_g$. 

We know that any irreducible character $\rho$ of $G$ is of the form $\eta \times \chi$, where $\eta$ is an irreducible character of $A$, and $\chi$ is a character of $T$.
Further, if $T$ is the $k$-dimensional torus then $\chi = \chi_1\times \chi_2 \times\cdots\times \chi_k$, where $\chi(e^{i\theta_1}, \dots, e^{i\theta_k}) = \prod \chi_j(e^{i\theta_j})$ and $\chi_j(e^{i\theta_j}) = e^{im_j\theta_j}$ with $m_j \in \mathbb{Z}$. 
The weight of $\chi$ is given by the $k$-tuple $(m_1,\ldots, m_k)$. 
Thus, the set of characters of $G$ with weight $\leq n$ are of the form $(\eta,\chi_1,\ldots, \chi_k)$, where $\chi_j(e^{i\theta}) = e^{i m_j\theta}$, with $|m_j| \leq n$ for all $j$, and $\eta \in Irr(A)$. 
Thus $Irr(G)_n = Irr(A)\times Irr(T)_n$.

Now, for $X = (a, g)\in G$ and $\rho \in Irr(G)_n$:
$$\lim_{n\rightarrow\infty}\left(\frac{1}{|A||Irr(T)_n|}\sum_{\rho \in Irr(A\times T)_n} \frac{\rho(X)}{\rho(1)}\right) = \lim_{n\rightarrow \infty} \frac{1}{|A|(2n+1)^k}\sum_{\eta \in Irr(A)}\sum_{\chi \in Irr(T)_n} \frac{\eta(a)\chi(g)}{\eta(1)}.$$
This is equal to 
$$\left(\frac{1}{|A|}\sum_{\eta \in Irr(A)}\frac{\eta(a)}{\eta(1)}\right) \left(\lim_{n\rightarrow \infty} \frac{1}{(2n+1)^k}\sum_{\chi \in Irr(T)_n} \chi(g)\right) .$$
This proves the result as we have already proved the result for finite group $A$ and the torus $T$. 
\end{proof}

We now prove a structure theorem for compact Lie FC groups. 

\begin{theorem}\label{StrLieFC}
Let $G$ be a compact Lie group which is an FC group and let $G^0$ denote the connected component of $G$ containing the identity. 
Then $G$ is of the form  $\displaystyle\frac{G^0\times \Delta}{N}$, where $\Delta$ is a finite group and $N$ is a normal subgroup of $G^0$ as well as to a normal subgroup of $\Delta$. 
Moreover, the embedding of $N$ in $G^0 \times \Delta$ is the diagonal embedding $n \mapsto (n, n^{-1})$ under those isomorphisms. 
\end{theorem} 

\begin{proof}
By Theorem 3.9 of \cite{HR}, a compact FC group $G$ is isomorphic to $\frac{G^0 \times \Delta}{N}$ where $\Delta$ is a profinite, totally disconnected group and $N$ consists of elements of the type $(g, g^{-1}) \in G^0 \times \Delta$.

The subgroup $N$ is thus isomorphic to a subgroup of $G^0$ and a subgroup of $\Delta$.
Further, if $G$ is a compact Lie FC group then by the structure theory of compact Lie groups $G^0$ is a central compact subgroup of $G$ and hence it is a torus. 
Since $N$ is isomorphic to a closed subgroup of $G^0$, it is the product of a torus (of possibly smaller dimension) and a finite abelian group. 
Since $N$ is a closed subgroup of $\Delta$, it is also profinite. 
Hence $N$ is finite. 

Now, $G/G^0$ is finite and is isomorphic to $\Delta/N$. Since $N$ is finite, $\Delta$ has to be finite. 
The proof is now complete. 
\end{proof}

We recall the statement of Theorem 1.2 here and give its proof. 
\vskip 0.5cm
\noindent{\bf Theorem 1.2.}
{\em 
Let ${\displaystyle G \cong \frac{G^0\times \Delta}{N}}$ be a compact Lie FC group. 
Then for $g_0\in G$, we have 
$$\mu^2 \big(\alpha^{-1}(g_0)\big) = \frac{|N|}{|\Delta|}\lim_{t \rightarrow \infty}\left(\frac{1}{|Irr(G^0)_t|}\sum_{\chi \in Irr(G)_t}\frac{\chi(g_0)}{\chi(1)}\right).$$
}

\begin{proof}
We know that $G\cong \frac{G^0\times \Delta}{N}$, where $\Delta$ is finite, and $N$ is a finite central subgroup of $G^0\times \Delta$. Denote $G^0\times \Delta$ by $H$. Let $q: H\rightarrow G$ denote the quotient map.

We claim that $\{hN\mid h \in H'\} = G'$. 
Let $a,b\in G$, $a = xN$, $b = yN$, then we have $[a,b] = [xN,yN] = xNyNx^{-1}Ny^{-1}N$, which is $xyx^{-1}y^{-1}N$, since $N \subseteq Z(H)$. 

Let $g_0 \in G'$, so $g_0 = h_0N$, where $h_0\in H'$. 
Now, $h_0 = (e, k_0)$, where $k_0 \in \Delta'$, since $G^0$ is abelian. 
As, $G$ is a quotient of $H$, the normalized Haar measure $\mu$ on $G$ is $\mu(Y) = \nu(q^{-1}(Y))$, where $\nu$ is the normalized Haar mesaure on $H$. 

Now, for $g_0\in G'$, the set $G^{(2)}_{g_0} = \{(a,b)\in G^2 \mid [a,b]= g_0 \}$. 
Then 
$$q^{-1}(G^{(2)}_{g_0}) = \displaystyle\sqcup_{n\in N} \{(x,y)\in H^2 \mid [x,y] = h_0n \} .$$
Hence $\mu^2(G^{(2)}_{g_0}) = \displaystyle\sum_{n\in N}\nu^2(\{(x,y)\in H^2 \mid [x,y] = h_0n \})$. 
Now, as $H = G^0\times \Delta$, where $G^0$ is a product of finite copies of $S^1$, and $\Delta$ is finite, by Lemma~\ref{LemSnFin}, we have:
$$\sum_{n\in N}\nu^2(\{(x,y)\in H^2 \mid [x,y] = h_0n \}) = \sum_{n \in N}\left( \lim_{t \rightarrow \infty}\frac{1}{|\Delta||Irr(G^0)_t|}\sum_{\chi \in Irr(H)_t}\frac{\chi(h_0n)}{\chi(1)}\right) .
$$
For $\chi \in Irr(H)$, let $\rho$ be the associated irreducible representation. 
Then by Schur's lemma $\rho(n) = \lambda_\chi(n) Id$, where $\lambda_\chi(n) \in \C$. 
Thus $\chi(h_0n)= \lambda_n\chi(h_0)$. As $N$ is finite, $\lambda_\chi(n)$ is a $k$th root of unity for some integer $k \mid |N|$. 
Hence we have the following: 
\begin{eqnarray*}
\mu^2(G^{(2)}_{g_0}) &=& \sum_{n \in N}\left( \lim_{t \rightarrow \infty}\frac{1}{|\Delta||Irr(G^0)_t|}\sum_{\chi \in Irr(H)_t}\frac{\chi(h_0n)}{\chi(1)}\right)\\
~&=& \lim_{t \rightarrow \infty}\left(\frac{1}{|\Delta||Irr(G^0)_t|}\sum_{n \in N}\sum_{\chi \in Irr(H)_t}\frac{\lambda_\chi(n)\chi(h_0)}{\chi(1)}\right)\\
~&=&\lim_{t \rightarrow \infty}\left(\frac{1}{|\Delta||Irr(G^0)_t|}\sum_{\chi \in Irr(H)_t}\frac{\left(\sum_{n \in N}\lambda_\chi(n)\right)\chi(h_0)}{\chi(1)}\right)\\
\end{eqnarray*}
Now, when $\chi|_N \neq \deg(\chi)$, $\lambda_\chi(n)$ is a non-trivial $k_\chi$th root of unity, for some positive integer $k_{\chi} \geq 2$. 
In this case, we have $\sum_{n\in N}\lambda_\chi(n) = \frac{|N|}{k_\chi}\sum_{l=0}^{k_\chi} e^{i2l\pi/k_\chi}  = 0$. 
Hence 
$$\mu^2(G^{(2)}_{g_0}) = \lim_{t \rightarrow \infty}\left(\frac{1}{|\Delta||Irr(G^0)_t|}\sum_{\substack{\chi \in Irr(H)_t\\ \chi|_N = \deg(\chi)}}\frac{\left(|N|\right)\chi(h_0)}{\chi(1)}\right)$$
As $N$ is finite abelian, any irreducible character of $H$ is $\chi_0 \times \eta$, where $\chi_0 \in Irr(G)$, and $\eta \in \hat{N}$. 
Thus $Irr(H)_t = \{\chi_0\times \eta \mid \chi_0\in Irr(G)_t,\eta\in \hat{N} \}$, thus $|Irr(H)_t| = |N||Irr(G)_t|$. 
Thus, 
\begin{eqnarray*}
\mu^2(G^{(2)}_{g_0}) &=& \lim_{t \rightarrow \infty}\left(\frac{1}{|\Delta||Irr(G^0)_t|}|N|\sum_{\chi \in Irr(G)_t}\frac{\chi(g_0)}{\chi(1)}\right)\\
&=& \frac{|N|}{|\Delta|}\lim_{t \rightarrow \infty}\left(\frac{1}{|Irr(G^0)_t|}\sum_{\chi \in Irr(G)_t}\frac{\chi(g_0)}{\chi(1)}\right).
\end{eqnarray*}
\end{proof}

\section{Compact Lie Groups with Open FC-Centre}
Let $G$ be a compact Lie group with an open FC centre $F$. 
We know from \cite{HR} that for a compact group $G$, $\mu(\alpha^{-1}(1)) > 0$ if and only if the FC centre $F$ is open in $G$. 
As an open subgroup of $G$, $F$ is of finite index in $G$. 

Now, for $g \notin F$, the conjugacy class of $g$ is an infinite set. 
Thus $\mu(\alpha^{-1}(g)) = 0$ for $g \notin F$.

Let $g \in F$. 
The set $\alpha^{-1}(g) = \{(x,y) : [x,y] = g\}$. 
For $x \in G$ we define the following subset of $G$: $$Z_G^g(x) = \{y \in G : [x,y] = g\}.$$ 

\begin{lemma}\label{LOFC1}
With the above notations, $\mu(Z^g_G(x))=\mu(Z_G(x))$.
\end{lemma}
\begin{proof}
Fix $h_0$, such that $xh_0x^{-1}h_0 = g$. 
Let $y \in Z_G(x)$. 
We see that:
$$[x,h_0y] = x(h_0y)x^{-1}y^{-1}h_0^{-1} = xh_0x^{-1}h_0^{-1} = [x, h_0]$$ 
since $y \in Z_G(x)$.
Thus $h_0.Z_G(x) \subseteq Z_G^g(x)$.
Let $z \in Z_G^g(g)$. Let $y = h_0^{-1}z$.
Then 
$$[x,y] = xh_0^{-1}zx^{-1}z^{-1}h_0 = xh_0^{-1}x^{-1}gh_0 = 1 .$$ 
Thus $h_0^{-1}z \in Z_G(x)$, $z \in h_0Z_G(x)$. 
This implies that $Z^g_G(X) \subseteq h_0Z_G(x)$ so $Z_G^g(x) = h_0Z_G(x)$.
Hence, $\mu(Z_G^g(x)) = \mu(Z_G(x))$.
\end{proof}

Now, for $x \notin F$, $Z_G(x)$, $\mu(Z_G(x)) = 0$, as $Z_G(x)$ has strictly lower dimension than $G$, $\mu(Z_G(x)) = 0.$ 
Thus 
$$\mu(\alpha^{-1}(g)) = \int_G \mu(Z_G(x)) dx = \int_{F} \mu(Z_G(x)) dx .$$ 
We now show the following:

\begin{theorem}
Let $G$ compact group with an open FC-centre $F$. 
For $g \in F$, $$\mu(\alpha^{-1}(g)) = \mu(\{(x,y) \in F^2: [x,y] = g\}).$$
\end{theorem}

\begin{proof}
We observe the following about $\alpha^{-1}(g)$
\begin{eqnarray*}
        \alpha^{-1}(g) &=& \bigsqcup_{x \in G}\{y \in G: [x,y] = g\}\\
        &=& \left(\bigsqcup_{x \in F}\{y \in G: [x,y] = g\}\right) \sqcup \left(\bigsqcup_{x \in G\setminus F}\{y \in G: [x,y] = g\}\right)\\
        &=& \left(\bigsqcup_{x \in F}\left(\{y \in F: [x,y] = g\}\right) \sqcup \{y \in G \setminus F: [x,y] = g\}\right) \sqcup A_{G \setminus F}\\
        &~& \text{ where }A_{G \setminus F} = \left(\bigsqcup_{x \in G\setminus F}\{y \in G: [x,y] = g\}\right)\\
        &=& \{(x,y) \in F \times F: [x,y] = g\} \sqcup\left(\bigsqcup_{x \in F}\{y \in G \setminus F: [x,y] = g\}\right) \sqcup A_{G \setminus F}\\
    \end{eqnarray*}
We see that $\mu(A_{G \setminus F}) = \int_{x \notin F} \mu(Z_G^g(x)) dx = 0$, by Lemma~\ref{LOFC1}. 
Now, we can rewrite the set $\bigsqcup_{x \in F}\{y \in G \setminus F: [x,y] = g\}$ as $\bigsqcup_{y \in G \setminus F}\{x \in F \setminus F: [x,y] = g\}$ which is a subset of $\bigsqcup_{y \in G \setminus F}\{x \in F \setminus G: [x,y] = g\}$, and for each $y \notin F$, $\{x \in F \setminus G: [x,y] = g\}$ is a set of measure 0. 
Thus, we see that $\bigsqcup_{y \in G \setminus F}\{x \in F \setminus G: [x,y] = g\}$ is also of measure 0. 
Hence, 
$$\alpha^{-1}(g) =\mu(\{(x,y) \in F^2: [x,y] = g\}). $$
Hence $\mu(\alpha^{-1}(g))$ is equal to $\mu(\alpha^{-1}(g) \cap F \times F)$.
\end{proof}

\section{The commutator probability}
Let $G$ be a finite group. 
Paul Erd\"{o}s initiated the probabilistic study of finite groups in a series of papers \cite{Er}.
One of the notions he introduced there is that of commuting probability on $G$. 

Let $\C$ denote the set of commuting pairs in $G \times G$, $\C = \{(g, h)\in G \times G: gh = hg\}$. 
The commuting probability of $G$ is then 
$$Pr(G) := \frac{|\C|}{|G \times G|} .$$
This is the probability that two randomly chosen elements of $G$ commute with each other. 

This concept has a natural generalisation to compact groups. 
Let $G$ be a compact group and $\mu$ be a normalised Haar measure on $G$, that is, we assume that $\mu(G) = 1$.
The subset $\C$ of commuting pairs in $G \times G$ is a closed subset of $G \times G$. 
Therefore it is a measurable subset of $G \times G$ and so $\mu(\C)$ makes sense. 
We then define, following \cite{Gu},
$$Pr(G) := \mu(\C) .$$
Note that, $\mu(G) = \mu(G \times G) = 1$, since $\mu$ is normalised. 

For finite groups, the notion of commuting probability has been generalised to study commutators. 
Let $g \in G$.
If $\C_g$ denotes the set $\{(a,b)\in G \times G: aba^{-1}b^{-1} = g\}$ then, following \cite{PS}, we define $Pr_G(g)$ by
$$Pr_G(g) := \frac{|\C_g|}{|G \times G|} .$$
This is the probability that two randomly chosen elements of $G$ have $g$ as their commutator. 
When $g = e$, the identity element of $G$, $Pr_G(e)$ is the commuting probability, $Pr(G)$ of $G$. 

The notion of commutator probability has a natural generalisation to compact groups.
Let $G$ be a compact group with a normalised Haar measure $\mu$ and fix a $g \in G$.
As above, let $\C_g$ denote the set $\{(a,b)\in G \times G: aba^{-1}b^{-1} = g\}$.
The set $\C_g$ is a closed subset of $G \times G$ and hence it is measurable. 
We then define $Pr_G(g)$ by 
$$Pr_G(g) = \mu(\C_g) .$$
This is the probability that two randomly chosen elements of $G$ have $g$ as their commutator. 

The commutator probability of a finite group $G$ is nicely related to complex irreducible representations of $G$, through the formula of Frobenius discussed earlier in this paper.
For $g \in G$
$$Pr_G(g) = \frac{1}{|G|}\sum_{\chi\in Irr(G)}\frac{\chi(g)}{\chi(1)} .$$

Our Theorem 1.1 gives a way to express the commutator probability of a compact group in terms of its characters. 

\begin{theorem}
Let $G$ be a connected, compact, Lie group with normalised Haar measure $\mu$.
Let $d$ denote the dimension of the group $G$ and let $r$ denote the rank of $G$. 
Let $Irr(G)_n$ denote the set of irreducible characters of $G$ whose weight-sum is less than or equal to $n$. 
Then for each $g \in G$, 
$$Pr_G(g)  = \mu\left(\alpha^{-1}(g)\right) = \lim_{n \to \infty} \frac{1}{|Irr(G)_n|^{d-r+1}} \sum_{\chi \in Irr(G)_n} \frac{\chi(g)}{\chi(1)} .$$
\end{theorem}


\begin{thebibliography}{99}
\bibitem[Er]{Er} Erd\"{o}s, P. Turan, P., {\em On some problems of a statistical group-theory IV}, Acta Math. Acad. Sci. Hungar. 19 (1968), 413--435
\bibitem[Gu]{Gu} Gustafson, W. H., {\em What is the probability that two group elements commute?}, Amer. Math. Monthly, {\bf 80}, (1973), 1031--1034.
\bibitem[HM]{HM} Hofmann, Karl H., Morris, Sidney A. {\em The Structure of Compact Groups}, (de Gruyter, Berlin, Second Edition, 2006).
\bibitem[HR]{HR} Hofmann, Karl H., Russo, Fransesco G. {\em The probability that $x$ and $y$ commute in a compact group}, Math. Proc. Camb. Phil. Soc. (2012), 153, 557--571
\bibitem[PS]{PS} Pournaki, M.R., Sobhani, R. {\em Probability that the commutator of two group elements is equal to a given element}, J. Pure Appl. Algebra (2007), 212, 727--734
\end{thebibliography}
\end{document}